\documentclass[a4paper,12pt]{amsart}

\usepackage{amsfonts}
\usepackage{amsmath}
\usepackage{amssymb}
\usepackage{graphicx}

\usepackage[usenames]{color}
\usepackage[colorlinks]{hyperref}

\newtheorem{thm}{Theorem}[section]
\newtheorem{lem}[thm]{Lemma}

\newtheorem{assum}[thm]{Assumption}

\theoremstyle{definition}

\theoremstyle{remark}
\newtheorem{rem}{Remark}[section]
\newtheorem{defn}{Definition}
\newtheorem{exam}[rem]{Example}
\numberwithin{equation}{section}

\begin{document}

\title[Heat equation with singular thermal conductivity]{Heat equation with singular thermal conductivity}

\author[M. Ruzhansky]{Michael Ruzhansky}
\address{
	Michael Ruzhansky:
	 \endgraf
  Department of Mathematics: Analysis, Logic and Discrete Mathematics
  \endgraf
  Ghent University, Krijgslaan 281, Building S8, B 9000 Ghent
  \endgraf
  Belgium
  \endgraf
  and
  \endgraf
  School of Mathematical Sciences
    \endgraf
    Queen Mary University of London
  \endgraf
  United Kingdom
	\endgraf
  {\it E-mail address} {\rm michael.ruzhansky@ugent.be}
}

\author[M. Sebih]{Mohammed Elamine Sebih}
\address{
  Mohammed Elamine Sebih:
  \endgraf
  Laboratory of Geomatics, Ecology and Environment (LGEO2E)
  \endgraf
  University Mustapha Stambouli of Mascara, 29000 Mascara
  \endgraf
  Algeria
  \endgraf
  {\it E-mail address} {\rm sebihmed@gmail.com, ma.sebih@univ-mascara.dz}
}

\author[N. Tokmagambetov ]{Niyaz Tokmagambetov }
\address{
  Niyaz Tokmagambetov:
  \endgraf 
  Centre de Recerca Matem\'atica
  \endgraf
  Edifici C, Campus Bellaterra, 08193 Bellaterra (Barcelona), Spain
  \endgraf
  and
  \endgraf   
  Institute of Mathematics and Mathematical Modeling
  \endgraf
  125 Pushkin str., 050010 Almaty, Kazakhstan
  \endgraf  
  {\it E-mail address:} {\rm tokmagambetov@crm.cat; tokmagambetov@math.kz}
  }

\thanks{This research was partly funded by the Committee of Science of the Ministry of Science and Higher Education of the Republic of Kazakhstan (Grant No. AP14872042). The authors are supported by the FWO Odysseus 1 grant G.0H94.18N: Analysis and Partial Differential Equations and by the Methusalem programme of the Ghent University Special Research Fund (BOF) (Grant number 01M01021). Michael Ruzhansky is also supported by EPSRC grant
EP/R003025/2.}

%, numerical analysis

\begin{abstract}
In this paper, we study the heat equation with an irregular spatially dependent thermal conductivity coefficient. We prove that it has a solution in an appropriate very weak sense. Moreover, the uniqueness result and consistency with the classical solution if the latter exists are shown. Indeed, we allow the coefficient to be a distribution with a toy example of a $\delta$-function.
\end{abstract}

\maketitle

%%%%%%%%%%%%%%%%%%%%%%%%%%%%%%%%%%%%%%%%%%%%%%%%%%%%%%%%%%%%%%%%%%%%%%%%%%%%%%%%%%%%%%%%%%%%%%%%%%%%%%%%%%%%%%%%%%%%%%%%%%%%%%%%%%%%%%%%%%%%%%%%%%%%%%%%%%%%%%%%%%%%%%%%%%%%%%%%%%%%%%%%%%%%%%%%%%%%

\section{Introduction}
In the present paper we investigate the Cauchy problem
\begin{equation}
    \left\lbrace
    \begin{array}{l}
    u_{t}(t,x) - \sum_{j=1}^{d}\partial_{x_{j}}\left(h(x)\partial_{x_{j}}u(t,x)\right)=0, \,\,\,(t,x)\in\left[0,T\right]\times \mathbb{R}^{d},\\
    u(0,x)=u_{0}(x), \,\,\, x\in\mathbb{R}^{d}, \label{Equation intro}
    \end{array}
    \right.
\end{equation}
where the coefficient $h$ is assumed to be singular and positive, that is, there exists $h_0 >0$ such that, $0< h_0\leq h$. This kind of problems arise naturally when modelling the heat transfer in nonhomogeneous mediums, see for instance \cite{WZW08}. In such case, the coefficient $h$ stands for the thermal conductivity of the material and $u$ represents the temperature function.

Nonhomogeneous or anisotropic materials are the most existing materials nowadays. They are made up of solid, liquid and gas, for instance porous capillary bodies and cellular systems. For such mediums, mechanical and thermal properties vary in time and position, this is why their thermal conductivity is often expressed as a function of the time, the position and the temperature. See for example \cite{AC08, AC10, AC14, Kam90}, the book \cite{WZW08} and \cite{TAKO96} for spatially dependent thermal conductivity.  

For phase-change problems \cite{LLG98, Vol90, VS93} or when studying the heat conduction in mediums made by different layers or in the presence of cracks \cite{Shi93, TVNDL18}, the properties of the material may change suddenly and it is natural to consider the thermal conductivity to be a non-regular function.

Here, we consider a singular thermal conductivity which depends only on the spatial variable and we study the well-posedness of the Cauchy problem (\ref{Equation intro}).

While from a physical point of view it is natural to formulate such problem, mathematically we can not even pose it in the case of distributional coefficients, in view of the Schwartz impossibility result about multiplication of distributions \cite{Sch54}. In order to provide a solution for this problem and to give a mathematical framework for numerical modelling, the authors in \cite{GR15} introduced the concept of very weak solutions for the analysis of second order hyperbolic equations with irregular coefficients. In order to show usefulness of the concept, the notion of very weak solutions was later applied in a series of papers, either for physical models or abstract mathematical problems \cite{ART19, MRT19, RT17a, RT17b, RY20, SW20}. In these papers, the authors dealt with equations with time-dependent coefficients. In \cite{Gar21, ARST21} 
the authors started using the concept of very weak solutions for equations with coefficients depending on the spatial variable. Recently, the authors in \cite{CRT22a, CRT22b} used the approach of very weak solutions for the heat equation and the fractional Schrödinger equation for general hypoelliptic operators in the setting of graded Lie groups.

Our aim in the present work, is to consider the Cauchy problem (\ref{Equation intro}) and prove that it is very weak well posed, the uniqueness is proved in an appropriate sense. Moreover, we prove that the theory of very weak solutions for our model is consistent with the classical theory.

There are many papers in the literature which are concerned with numerical simulations of the problem of the kind (\ref{Equation intro}), we refer to \cite{AC08, AC10, AC14, LLG98, SPG02, Vol90, VS93} to cite only few. 

Essentially, the present work may be considered as an extension of the study of the well-posedness of the heat equation by incorporating coefficients having strong singularities, distributional or non-distributional. This is possible within the framework of the theory of very weak solutions which gives a meaningful notion of solution when the classical theory fails. On the other hand, the concept of very weak solutions seems to be well adapted to numerical simulations as it was pointed out by the authors in the papers \cite{ART19, MRT19, ARST21} 
for different models. Indeed, we are able to talk about uniqueness of very weak solutions in some sense.

\section{Main results}
For $T>0$ and $d\geq 1$, we consider the Cauchy problem for the divergence form of the heat equation 
\begin{equation}
    \left\lbrace
    \begin{array}{l}
    u_{t}(t,x) - \sum_{j=1}^{d}\partial_{x_{j}}\left(h(x)\partial_{x_{j}}u(t,x)\right)=0, \,\,\,(t,x)\in\left[0,T\right]\times \mathbb{R}^{d},\\
    u(0,x)=u_{0}(x), \,\,\, x\in\mathbb{R}^{d},
    \end{array}
    \right.
\label{Equation singular}
\end{equation}
where $h$ is positive and singular (singular thermal conductivity). Due to the singularities in the equation we introduce a notion of very weak solutions adapted to our problem. Moreover, we prove the existence, uniqueness, and consistency results in some appropriate sense.

\subsection{Existence of a very weak solution}
In this subsection, we want to prove that the equation \ref{Equation singular} has a very weak solution. To start with, we regularise the coefficient $h$ and the Cauchy data $u_0$ by convolution with a suitable mollifier $\psi$, generating families of smooth functions $(h_{\varepsilon})_{\varepsilon}$ and $(u_{0,\varepsilon})_{\varepsilon}$, that is
\begin{equation}
    h_{\varepsilon}(x) = h\ast \psi_{\varepsilon}(x)
\end{equation}
and
\begin{equation}
    u_{0,\varepsilon}(x) = u_0\ast \psi_{\varepsilon}(x),
\end{equation}
where
\begin{equation}
    \psi_{\varepsilon}(x) = \varepsilon^{-1}\psi(x/\varepsilon),\,\,\,\varepsilon\in\left(0,1\right].
\end{equation}
The function $\psi$ is a Friedrichs-mollifier, i.e. $\psi\in C_{0}^{\infty}(\mathbb{R}^{d})$, $\psi\geq 0$ and $\int\psi =1$.

Throughout this paper, we will write $f\lesssim g$ for two functions $f$ and $g$ on the same domain, if there
exists a positive constant $C$ such that $f\leq Cg$. We denote by
\begin{equation*}        
H^{k}(\mathbb{R}^d):=\left\{ f \text{\,\,is\,\,measurable}:\, \Vert f\Vert_{H^{k}}:= \Vert f\Vert_{L^{2}} + \Vert \nabla^{k} f\Vert_{L^{2}} < +\infty \right\},   
\end{equation*}
for each $k\in\mathbb Z_{+}$.
Also, $W^{1,\infty}(\mathbb{R}^d)$ is a Banach space defined by
    \begin{equation*}
        W^{1,\infty}(\mathbb{R}^d):=\left\{ f \text{\,\,is\,\,measurable}:\, \Vert f\Vert_{W^{1,\infty}}:= \Vert f\Vert_{L^{\infty}} + \Vert \nabla f\Vert_{L^{\infty}} < +\infty \right\}.
    \end{equation*}

\begin{assum}\label{assumption}
On the regularisations of the coefficient $h$ and the data $u_0$, we assume that
\begin{equation}
    \inf_{\varepsilon\in (0,1]}\inf_{x\in \mathbb{R}^d}h_{\varepsilon}(x)>0,
\end{equation}
as well as there exist $N_0, N_1\in \mathbb{N}_0$ such that
\begin{equation}
     \Vert h_{\varepsilon}\Vert_{W^{1,\infty}} \lesssim \varepsilon^{-N_0}, \label{assum coeff}
\end{equation}
and
\begin{equation}
    \Vert u_{0,\varepsilon}\Vert_{H^2} \lesssim \varepsilon^{-N_1}. \label{assum data}
\end{equation}
\end{assum}

\begin{rem}
We note that such assumptions are natural for distributions. Indeed, by the structure theorems for distributions (see, e.g. \cite{FJ98}), we know that every compactly supported distribution can be represented by a
finite sum of (distributional) derivatives of continuous functions. Precisely, for $T\in \mathcal{E}'(\mathbb{R}^{d})$ we can find $n\in \mathbb{N}$ and functions $f_{\alpha}\in C(\mathbb{R}^{d})$ such that, $T=\sum_{\vert \alpha\vert \leq n}\partial^{\alpha}f_{\alpha}$. The convolution of $T$ with a mollifier gives
\begin{equation}
    T\ast\psi_{\varepsilon}=\sum_{\vert \alpha\vert \leq n}\partial^{\alpha}f_{\alpha}\ast\psi_{\varepsilon}=\sum_{\vert \alpha\vert \leq n}f_{\alpha}\ast\partial^{\alpha}\psi_{\varepsilon}=\sum_{\vert \alpha\vert \leq n}\varepsilon^{-\vert\alpha\vert}f_{\alpha}\ast\left(\varepsilon^{-1}\partial^{\alpha}\psi(x/\varepsilon)\right).
\end{equation}
Using an appropriate norm, we see that the regularisation of $T$ satisfies the above assumption. Nevertheless, let us give some more examples.
\end{rem}

\begin{exam} \label{example moderateness}
\leavevmode
\begin{itemize}
\item[(1)] Let $f=\delta_{0}$. We have
$f_{\varepsilon}(x) = \varepsilon^{-1}\psi(\varepsilon^{-1}x)\leq C\varepsilon^{-1}.$
\item[(2)] For $f=\delta_{0}^{2}$, we take
$f_{\varepsilon}(x) = \varepsilon^{-2}\psi^{2}(\varepsilon^{-1}x) \leq C\varepsilon^{-2}.$
\end{itemize}
\end{exam}

\begin{defn}[Moderateness] \label{defn:Moderatness}
\leavevmode
\begin{itemize}
    \item[(i)]  A net of functions $(f_{\varepsilon})_{\varepsilon}$, is said to be $W^{1,\infty}$-moderate, if there exist $N\in\mathbb{N}_{0}$ such that
\begin{equation*}
    \Vert f_{\varepsilon}\Vert_{W^{1,\infty}} \lesssim \varepsilon^{-N}.
\end{equation*}
    \item[(ii)] A net of functions $(g_{\varepsilon})_{\varepsilon}$, is said to be $H^2$-moderate, if there exist $N\in\mathbb{N}_{0}$ such that
\begin{equation*}
    \Vert g_{\varepsilon}\Vert_{H^2} \lesssim \varepsilon^{-N}.
\end{equation*}
    \item[(iii)] A net of functions $(u_{\varepsilon})_{\varepsilon}$ from $C([0,T]; H^{2}(\mathbb R^d))$ is said to be $C$-moderate, if there exist $N\in\mathbb{N}_{0}$ such that
\begin{equation*}
    \sup_{t\in[0,T]}\Vert u_{\varepsilon}(t,\cdot)\Vert_{H^2} \lesssim \varepsilon^{-N}.
\end{equation*}
\end{itemize}
\end{defn}

\begin{rem}
We remark that the regularisations of the coefficient $h$ and the Cauchy data $u_0$ are moderate by assumption.
\end{rem}

\begin{defn}[Very weak solution] \label{defn:very weak sol}
The net $(u_{\varepsilon})_{\varepsilon}\in C([0,T]; H^{2}(\mathbb R^d))$ is said to be a very weak solution to the Cauchy problem (\ref{Equation singular}), if there exist a $W^{1,\infty}$-moderate regularisation of the coefficient $h$ and a $H^2$-moderate regularisation of $u_0$, such that the family $(u_{\varepsilon})_{\varepsilon}$ solves the regularized problem
\begin{equation}
    \left\lbrace
    \begin{array}{l}
    \partial_{t}u_{\varepsilon}(t,x) - \sum_{j=1}^{d}\partial_{x_{j}}\left(h_{\varepsilon}(x)\partial_{x_{j}}u_{\varepsilon}(t,x)\right)=0, \,\,\,(t,x)\in\left[0,T\right]\times \mathbb{R}^{d},\\
    u_{\varepsilon}(0,x)=u_{0,\varepsilon}(x), \,\,\, x\in\mathbb{R}^{d}, \label{Equation regularized}
    \end{array}
    \right.
\end{equation}
for all $\varepsilon\in\left(0,1\right]$, and is $C$-moderate.
\end{defn}

\begin{thm}[Existence]
\label{TH: Exist}
Let the thermal conductivity $h$ be positive and assume that the regularisations of $h$ and $u_0$ satisfy the conditions of Assumption \ref{assumption}. Then, the Cauchy problem (\ref{Equation singular}) has a very weak solution.
\end{thm}

The following lemma is a key to the proof of the existence of a very weak solution to our model problem, Theorem \ref{TH: Exist}. It is stated in the case when $h$ is a regular function.

\begin{lem} \label{lem 1}
Let $h\in L^{\infty}\left({\mathbb{R}^d}\right)$ be positive. Assume that $u_0 \in H^{1}(\mathbb{R}^d)$.
Then, the unique solution $u\in C([0, T]; H^{1}(\mathbb R^d))$ to the Cauchy problem (\ref{Equation singular}), satisfies the estimates
\begin{equation}
    \Vert u(t,\cdot)\Vert_{L^2} \lesssim \left(1 + \Vert h\Vert_{L^{\infty}}\right)^{\frac{1}{2}} \Vert u_0\Vert_{H^1}, \label{Energy estimate}
\end{equation}
and
\begin{equation}
    \Vert \partial_{x_i}u(t,\cdot)\Vert_{L^2} \lesssim \left(1 + \Vert h\Vert_{L^{\infty}}\right) \Vert u_0\Vert_{H^1}, \label{Energy estimate 1}
\end{equation}
for all $i=1,...,d$ and $t\in [0,T]$. Moreover, if $h\in W^{1,\infty}(\mathbb{R}^d)$ and $u_0 \in H^{2}(\mathbb{R}^d)$, then, the solution $u\in C([0, T]; H^{2}(\mathbb R^d))$ and satisfies the estimate
\begin{equation}
    \Vert \Delta u(t,\cdot)\Vert_{L^2} \lesssim \left(2 + \Vert h\Vert_{W^{1,\infty}}\right)^2 \Vert u_0\Vert_{H^2}, \label{Energy estimate 2}
\end{equation}
for all $t\in [0,T]$.
\end{lem}

\begin{proof}
On the one hand, multiplying the equation in (\ref{Equation singular}) by $u_t$ and integrating with respect to the variable $x$ over $\mathbb{R}^d$ and taking the real part, we get
\begin{equation}
    Re \left( \langle u_{t}(t,\cdot),u_{t}(t,\cdot)\rangle_{L^2} - \sum_{j=1}^{d}\langle \partial_{x_j} (h(\cdot)\partial_{x_j}u(t,\cdot)),u_{t}(t,\cdot)\rangle_{L^2} \right) = 0, \label{Energy functional 1}
\end{equation}
where $\langle \cdot, \cdot \rangle_{L^2}$ denotes the inner product in $L^2(\mathbb{R}^{d})$. After short calculations, we easily see that
\begin{equation}
    Re \langle u_{t}(t,\cdot),u_{t}(t,\cdot)\rangle_{L^2} = \Vert u_t(t,\cdot)\Vert_{L^2}^2
\end{equation}
and
\begin{equation}
    Re \sum_{j=1}^{d}\langle \partial_{x_j} (h(\cdot)\partial_{x_j}u(t,\cdot)),u_{t}(t,\cdot)\rangle_{L^2} = -\frac{1}{2}\partial_{t}\sum_{j=1}^{d} \Vert h^{\frac{1}{2}}\partial_{x_j}u(t,\cdot)\Vert_{L^2}^2.
\end{equation}
From (\ref{Energy functional 1}), it follows that
\begin{equation}
    \Vert u_t(t,\cdot)\Vert_{L^2}^2 + \frac{1}{2}\partial_{t}\sum_{j=1}^{d} \Vert h^{\frac{1}{2}}\partial_{x_j}u(t,\cdot)\Vert_{L^2}^2 = 0. \label{Energy 1}
\end{equation}
On the other hand, if we multiply the equation in (\ref{Equation singular}) by $u$ and integrate over $\mathbb{R}^d$, we obtain
\begin{equation}
    Re \left( \langle u_{t}(t,\cdot),u(t,\cdot)\rangle_{L^2} - \sum_{j=1}^{d}\langle \partial_{x_j} (h(\cdot)\partial_{x_j}u(t,\cdot)),u(t,\cdot)\rangle_{L^2} \right) = 0. \label{Energy functional 2}
\end{equation}
Again, after short calculations we get
\begin{equation}
    Re \langle u_{t}(t,\cdot),u(t,\cdot)\rangle_{L^2} = \frac{1}{2}\partial_{t}\Vert u(t,\cdot)\Vert_{L^2}^2,
\end{equation}
and
\begin{equation}
    Re \sum_{j=1}^{d}\langle \partial_{x_j} (h(\cdot)\partial_{x_j}u(t,\cdot)),u(t,\cdot)\rangle_{L^2} = -\sum_{j=1}^{d} \Vert h^{\frac{1}{2}}\partial_{x_j}u(t,\cdot)\Vert_{L^2}^2.
\end{equation}
It follows from (\ref{Energy functional 2}) that
\begin{equation}
    \frac{1}{2}\partial_{t}\Vert u(t,\cdot)\Vert_{L^2}^2 + \sum_{j=1}^{d} \Vert h^{\frac{1}{2}}\partial_{x_j}u(t,\cdot)\Vert_{L^2}^2 = 0. \label{Energy 2}
\end{equation}
By summing (\ref{Energy 1}) and (\ref{Energy 2}), we arrive at
\begin{equation}
    \frac{1}{2}\partial_{t}\left[ \Vert u(t,\cdot)\Vert_{L^2}^2 + \sum_{j=1}^{d} \Vert h^{\frac{1}{2}}\partial_{x_j}u(t,\cdot)\Vert_{L^2}^2 \right] = - \left[ \Vert u_t(t,\cdot)\Vert_{L^2}^2 + \sum_{j=1}^{d} \Vert h^{\frac{1}{2}}\partial_{x_j}u(t,\cdot)\Vert_{L^2}^2 \right]. \label{Energy}
\end{equation}
Let us denote by
\begin{equation}
    E(t):=  \Vert u(t,\cdot)\Vert_{L^2}^2 + \sum_{j=1}^{d} \Vert h^{\frac{1}{2}}\partial_{x_j}u(t,\cdot)\Vert_{L^2}^2.
\end{equation}
The right hand side in (\ref{Energy}) is negative, so that $\partial_{t}E(t) \leq 0$ and thus $E(t) \leq E(0)$ for all $t\in [0,T]$.
It follows that
\begin{equation}
    \Vert u(t,\cdot)\Vert_{L^2}^2 \leq \Vert u_0\Vert_{L^2}^2 + \sum_{j=1}^{d} \Vert h^{\frac{1}{2}}\partial_{x_j}u_0\Vert_{L^2}^2.
\end{equation}
Noting that for all $i=1,...,d$, the term $\Vert h^{\frac{1}{2}}\partial_{x_i}u_0\Vert_{L^2}^2$ can be estimated by
\begin{equation}
    \Vert h^{\frac{1}{2}}\partial_{x_i}u_0\Vert_{L^2}^2 \leq \Vert h\Vert_{L^{\infty}}\Vert u_0\Vert_{H^1}^2,
\end{equation}
we obtain the desired estimate for $u$, that is
\begin{equation}
    \Vert u(t,\cdot)\Vert_{L^2} \lesssim \left(1 + \Vert h\Vert_{L^{\infty}}\right)^{\frac{1}{2}} \Vert u_0\Vert_{H^1}.
\end{equation}
%where $\left(1 + \Vert h\Vert_{L^{\infty}}\right)$ is estimated by $\left(1 + \Vert h\Vert_{L^{\infty}}\right)^2$.
Now, the equality (\ref{Energy 1}) implies that
\begin{equation}
     \frac{1}{2}\partial_{t}\sum_{j=1}^{d} \Vert h^{\frac{1}{2}}\partial_{x_j}u(t,\cdot)\Vert_{L^2}^2 = -\Vert u_t(t,\cdot)\Vert_{L^2}^2.
\end{equation}
As the right hand side is negative, it follows that $\sum_{j=1}^{d} \Vert h^{\frac{1}{2}}\partial_{x_j}u(t,\cdot)\Vert_{L^2}^2$ is decreasing and thus,
\begin{equation}
    \Vert h^{\frac{1}{2}}\partial_{x_i}u(t,\cdot)\Vert_{L^2} \lesssim \left(1 + \Vert h\Vert_{L^{\infty}}\right) \Vert u_0\Vert_{H^1},
\end{equation}
for all $i=1,...,d$. Then, by using the assumption that $h$ is bounded from below, we obtain our estimate for $\partial_{x_i}u$. That is
\begin{equation}
    \Vert \partial_{x_i}u(t,\cdot)\Vert_{L^2} \lesssim \left(1 + \Vert h\Vert_{L^{\infty}}\right) \Vert u_0\Vert_{H^1}, \label{Estimate dxjU}
\end{equation}
for all $i=1,...,d$ and $t\in [0,T]$. Let us now assume that $\nabla h\in L^{\infty}\left({\mathbb{R}^d}\right)$, $u_0 \in H^{2}(\mathbb{R}^d)$ and prove the estimate for $\Delta u$. On the one hand, using the equality (\ref{Energy 2}), we easily show that
\begin{equation}
    \Vert u(t,\cdot)\Vert_{L^2} \lesssim \Vert u_0\Vert_{L^2}, \label{Estimate}
\end{equation}
for all $t\in [0,T]$. On the other hand, we know that if $u$ solves the Cauchy problem
\begin{equation}
    \left\lbrace
    \begin{array}{l}
    u_{t}(t,x) - \sum_{j=1}^{d}\partial_{x_{j}}\left(h(x)\partial_{x_{j}}u(t,x)\right)=0, \,\,\,(t,x)\in\left[0,T\right]\times \mathbb{R}^{d},\\
    u(0,x)=u_{0}(x), \,\,\, x\in\mathbb{R}^{d},
    \end{array}
    \right.
\end{equation}
then $u_t$ solves
\begin{equation}
    \left\lbrace
    \begin{array}{l}
   \partial_t u_{t}(t,x) - \sum_{j=1}^{d}\partial_{x_{j}}\left(h(x)\partial_{x_{j}}u_t(t,x)\right)=0, \,\,\,(t,x)\in\left[0,T\right]\times \mathbb{R}^{d},\\
    u_t(0,x)=\sum_{j=1}^{d}\partial_{x_{j}}\left(h(x)\partial_{x_{j}} u_0(x)\right), \,\,\, x\in\mathbb{R}^{d}.
    \end{array}
    \right.
\end{equation}
Using (\ref{Estimate}), we get
\begin{align}
    \Vert u_t\Vert_{L^2} & \lesssim \sum_{j=1}^{d}\Vert \partial_{x_{j}}\left(h(\cdot)\partial_{x_{j}} u_0(\cdot)\right)\Vert_{L^2}\nonumber \\
    & \lesssim \sum_{j=1}^{d}\Vert \partial_{x_{j}} h(\cdot)\partial_{x_{j}} u_0(\cdot)\Vert_{L^2} + \sum_{j=1}^{d}\Vert h(\cdot)\partial_{x_{j}}^{2} u_0(\cdot)\Vert_{L^2}\nonumber \\ 
    & \lesssim \Vert \nabla h\Vert_{L^{\infty}}\Vert u_0\Vert_{H^1} + \Vert h\Vert_{L^{\infty}}\Vert u_0\Vert_{H^2}. \label{Estimate u_t}
\end{align}
The estimate for $\Delta u$ follows by taking the $L^2$ norm in the equality
\begin{equation}
    \sum_{j=1}^{d}h(x)\partial_{x_{j}}^{2}u(t,x) = u_{t}(t,x) + \sum_{j=1}^{d}\partial_{x_j}h(x)\partial_{x_{j}}u(t,x),
\end{equation}
obtained from the equation in (\ref{Equation singular}), and by using the so far proved estimates \eqref{Estimate dxjU} and \eqref{Estimate u_t} for $\partial_{x_j}u$ and $u_t$ and taking into consideration that $h$ is bounded from below, %\textcolor{blue}{
that is,
\begin{equation*}
    \Vert h(\cdot)\Delta u\Vert_{L^2} \lesssim \Vert u_t(t,\cdot)\Vert_{L^2} + \Vert \sum_{j=1}^{d}\partial_{x_j}h(\cdot)\partial_{x_{j}}u(t,\cdot)\Vert_{L^2}. 
\end{equation*}
The first term in the right hand side is estimated by $\Vert h\Vert_{W^{1,\infty}}\Vert u_0\Vert_{H^2}$ and the second term can be estimated as follows
\begin{align*}
    \Vert \sum_{j=1}^{d}\partial_{x_j}h(x)\partial_{x_{j}}u(t,x)\Vert_{L^2} & \lesssim \sum_{j=1}^{d} \Vert \partial_{x_j}h(\cdot)\Vert_{L^{\infty}}\Vert \partial_{x_{j}}u(t,\cdot)\Vert_{L^2}\\
    & \lesssim \Vert h\Vert_{W^{1,\infty}}\sum_{j=1}^{d}\Vert \partial_{x_{j}}u(t,\cdot)\Vert_{L^2}.
\end{align*}
From \eqref{Estimate dxjU}, using that for all $j=1,...,d$ and $t\in [0,T]$,
\begin{equation*}
    \Vert \partial_{x_j}u(t,\cdot)\Vert_{L^2} \lesssim \left(1 + \Vert h\Vert_{L^{\infty}}\right) \Vert u_0\Vert_{H^1} \lesssim \left(1 + \Vert h\Vert_{W^{1,\infty}}\right) \Vert u_0\Vert_{H^2},
\end{equation*}
we get that
\begin{align*}
    \Vert h(\cdot)\Delta u\Vert_{L^2} \lesssim \Vert h\Vert_{W^{1,\infty}}\left(2+\Vert h\Vert_{W^{1,\infty}}\right)\Vert u_0\Vert_{H^2}.
\end{align*}
The desired estimate follows from the estimate
\begin{equation*}
    \Vert h\Vert_{W^{1,\infty}}\left(2+\Vert h\Vert_{W^{1,\infty}}\right) \leq \left(2+\Vert h\Vert_{W^{1,\infty}}\right)^2
\end{equation*}
and the assumption that $h$ is bounded from below.
%} 
This ends the proof of the lemma.
\end{proof}

\begin{proof}[Proof of Theorem \ref{TH: Exist}]
Using the energy estimates (\ref{Energy estimate}), (\ref{Energy estimate 1}), (\ref{Energy estimate 2}) and the moderateness assumptions (\ref{assum coeff}) and (\ref{assum data}), of the coefficient $h$ and the data $u_0$, we arrive at
\begin{equation}
    \Vert u_{\varepsilon}(t,\cdot)\Vert_{H^2} \lesssim \varepsilon^{-N_0 - N_1},
\end{equation}
for all $t\in[0,T]$. That means that $(u_{\varepsilon})_{\varepsilon}$, the net of solutions to the regularized Cauchy problem (\ref{Equation regularized}) is $C$-moderate and the existence of a very weak solution follows.
\end{proof}

In the next theorems, we want to prove the uniqueness of the very weak solution and a consistency result. To do, we recall the following estimate which was implicitly proved in Lemma \ref{lem 1}.

\begin{lem} \label{lem 2}
Let $h\in L^{\infty}\left({\mathbb{R}^d}\right)$ be positive. Assume that $u_0 \in H^{1}(\mathbb{R}^d)$.
Then, the estimate
\begin{equation}
    \Vert u(t,\cdot)\Vert_{L^2} \lesssim \Vert u_0\Vert_{L^2}, \label{Energy estimate 3}
\end{equation}
for all $t\in [0,T]$, holds for the unique solution $u\in C([0, T]; L^{2}(\mathbb R^d))$ to the Cauchy problem (\ref{Equation singular}).
\end{lem}

\begin{proof}
Multiplying the equation in (\ref{Equation singular}) by $u$, integrating over $\mathbb{R}^d$ and taking the real part, we get
\begin{equation}
    Re \left( \langle u_{t}(t,\cdot),u(t,\cdot)\rangle_{L^2} - \sum_{j=1}^{d}\langle \partial_{x_j} (h(\cdot)\partial_{x_j}u(t,\cdot)),u(t,\cdot)\rangle_{L^2} \right) = 0. \label{Energy functional 3}
\end{equation}
We have that
\begin{equation}
    Re \langle u_{t}(t,\cdot),u(t,\cdot)\rangle_{L^2} = \frac{1}{2}\partial_{t}\Vert u(t,\cdot)\Vert_{L^2}^2,
\end{equation}
and
\begin{equation}
    Re \sum_{j=1}^{d}\langle \partial_{x_j} (h(\cdot)\partial_{x_j}u(t,\cdot)),u(t,\cdot)\rangle_{L^2} = -\sum_{j=1}^{d} \Vert h^{\frac{1}{2}}\partial_{x_j}u(t,\cdot)\Vert_{L^2}^2.
\end{equation}
From (\ref{Energy functional 3}), it follows that
\begin{equation}
    \frac{1}{2}\partial_{t}\Vert u(t,\cdot)\Vert_{L^2}^2 = -\sum_{j=1}^{d} \Vert h^{\frac{1}{2}}\partial_{x_j}u(t,\cdot)\Vert_{L^2}^2. \label{Energy 3}
\end{equation}
The right hand side in (\ref{Energy 3}) in negative, which means that $\Vert u(t,\cdot)\Vert_{L^2}^2$ is decreasing in time. This concludes the proof of the lemma.
\end{proof}

\subsection{Uniqueness of very weak solutions}
We prove the uniqueness of the very weak solution to the Cauchy problem (\ref{Equation singular}) in the sense of the following stability definition.

\begin{defn}[Uniqueness]
We say that the Cauchy problem (\ref{Equation singular}) has a unique very weak solution, if for all families of regularisations $(h_{\varepsilon})_{\varepsilon}$, $(\Tilde{h}_{\varepsilon})_{\varepsilon}$ and $(u_{0,\varepsilon})_{\varepsilon}$, $(\Tilde{u}_{0,\varepsilon})_{\varepsilon}$ of the coefficient $h$ and the Cauchy data $u_0$, satisfying
\begin{equation}
    \Vert h_{\varepsilon}-\Tilde{h}_{\varepsilon}\Vert_{W^{1,\infty}}\leq C_{k}\varepsilon^{k} \text{\,\,for all\,\,} k>0,
\end{equation}
and
\begin{equation}
    \Vert u_{0,\varepsilon}-\Tilde{u}_{0,\varepsilon}\Vert_{L^{2}}\leq C_{m}\varepsilon^{m} \text{\,\,for all\,\,} m>0,
\end{equation}
we have
\begin{equation*}
    \Vert u_{\varepsilon}(t,\cdot)-\Tilde{u}_{\varepsilon}(t,\cdot)\Vert_{L^{2}} \leq C_{N}\varepsilon^{N}
\end{equation*}
for all $N>0$,  
where $(u_{\varepsilon})_{\varepsilon}$ and $(\Tilde{u}_{\varepsilon})_{\varepsilon}$ are the families of solutions to the related regularized Cauchy problems.
\end{defn}

\begin{thm}[Uniqueness] \label{thm uniqueness}
Let $T>0$. Assume that the regularisations of the coefficient $h$ and the Cauchy data $u_0$ satisfy Assumption \ref{assumption}. Then, the very weak solution to the Cauchy problem (\ref{Equation singular}) is unique.
\end{thm}

\begin{proof}
Let $(h_{\varepsilon}, u_{0,\varepsilon})_{\varepsilon}$, $(\Tilde{h}_{\varepsilon}, \Tilde{u}_{0,\varepsilon})_{\varepsilon}$ be regularisations of the coefficient $h$ and the Cauchy data $u_0$ and let assume that
\begin{equation*}
    \Vert h_{\varepsilon}-\Tilde{h}_{\varepsilon}\Vert_{W^{1,\infty}}\leq C_{k}\varepsilon^{k} \text{\,\,for all\,\,} k>0,
\end{equation*}
and
\begin{equation*}
    \Vert u_{0,\varepsilon}-\Tilde{u}_{0,\varepsilon}\Vert_{L^{2}}\leq C_{m}\varepsilon^{m} \text{\,\,for all\,\,} m>0.
\end{equation*}
Let us denote by $U_{\varepsilon}(t,x):=u_{\varepsilon}(t,x)-\Tilde{u}_{\varepsilon}(t,x)$, where $(u_{\varepsilon})_{\varepsilon}$ and $(\Tilde{u}_{\varepsilon})_{\varepsilon}$ are the solutions to the families of regularized Cauchy problems, related to the families $(h_{\varepsilon}, u_{0,\varepsilon})_{\varepsilon}$ and $(\Tilde{h}_{\varepsilon}, \Tilde{u}_{0,\varepsilon})_{\varepsilon}$. Easy calculations show that $U_{\varepsilon}$ solves the Cauchy problem
\begin{equation}
    \left\lbrace
    \begin{array}{l}
    \partial_{t}U_{\varepsilon}(t,x) - \sum_{j=1}^{d}\partial_{x_{j}}\left(\Tilde{h}_{\varepsilon}(x)\partial_{x_{j}}U_{\varepsilon}(t,x)\right) = f_{\varepsilon}(t,x), \,\,\,(t,x)\in\left[0,T\right]\times \mathbb{R}^{d},\\
    U_{\varepsilon}(0,x)=(u_{0,\varepsilon}-\Tilde{u}_{0,\varepsilon})(x), \,\,\, x\in\mathbb{R}^{d},
    \end{array}
    \right.
\end{equation}
where
\begin{equation}
    f_{\varepsilon}(t,x) = \sum_{j=1}^{d}\partial_{x_{j}}\left[\left(h_{\varepsilon}(x)-\Tilde{h}_{\varepsilon}(x)\right)\partial_{x_{j}}u_{\varepsilon}(t,x)\right].
\end{equation}
By Duhamel's principle (see, e.g. \cite{ER18}), we obtain the following representation
\begin{equation}
    U_{\varepsilon}(x,t)=V_{\varepsilon}(x,t) + \int_{0}^{t}W_{\varepsilon}(x,t-s;s)ds \label{Representation U_eps}
\end{equation}
for $U_{\varepsilon}$, where $V_{\varepsilon}(x,t)$ is the solution to the homogeneous problem
\begin{equation}
    \left\lbrace
    \begin{array}{l}
    \partial_{t}V_{\varepsilon}(t,x) - \sum_{j=1}^{d}\partial_{x_{j}}\left(\Tilde{h}_{\varepsilon}(x)\partial_{x_{j}}V_{\varepsilon}(t,x)\right) = 0, \,\,\,(t,x)\in\left[0,T\right]\times \mathbb{R}^{d},\\
    V_{\varepsilon}(0,x)=(u_{0,\varepsilon}-\Tilde{u}_{0,\varepsilon})(x), \,\,\, x\in\mathbb{R}^{d},
    \end{array}
    \right.
\end{equation}
and $W_{\varepsilon}(x,t;s)$ solves
\begin{equation}
    \left\lbrace
    \begin{array}{l}
    \partial_{t}W_{\varepsilon}(t,x;s) - \sum_{j=1}^{d}\partial_{x_{j}}\left(\Tilde{h}_{\varepsilon}(x)\partial_{x_{j}}W_{\varepsilon}(t,x;s)\right) = 0, \,\,\,(t,x)\in\left[0,T\right]\times \mathbb{R}^{d},\\
    W_{\varepsilon}(0,x;s)=f_{\varepsilon}(s,x), \,\,\, x\in\mathbb{R}^{d}.
    \end{array}
    \right.
\end{equation}
Taking the $L^2$ norm in both sides of (\ref{Representation U_eps}) and using (\ref{Energy estimate 3}) to estimate $V_{\varepsilon}$ and $W_{\varepsilon}$, we obtain
\begin{align}
    \Vert U_{\varepsilon}(\cdot,t)\Vert_{L^2} & \leq \Vert V_{\varepsilon}(\cdot,t)\Vert_{L^2} + \int_{0}^{T}\Vert W_{\varepsilon}(\cdot,t-s;s)\Vert_{L^2} ds \nonumber\\
    & \lesssim \Vert u_{0,\varepsilon}-\Tilde{u}_{0,\varepsilon}\Vert_{L^2} + \int_{0}^{T}\Vert f_{\varepsilon}(s,\cdot)\Vert_{L^2} ds. \label{Estimate U_eps}
\end{align}
Let us estimate $\Vert f_{\varepsilon}(s,\cdot)\Vert_{L^2}$. We have
\begin{align*}
    \Vert f_{\varepsilon}(s,\cdot)\Vert_{L^2} & = \Vert\sum_{j=1}^{d}\partial_{x_{j}}\left[\left(h_{\varepsilon}(\cdot)-\Tilde{h}_{\varepsilon}(\cdot)\right)\partial_{x_{j}}u_{\varepsilon}(s,\cdot)\right]\Vert_{L^2}\\
    & \leq \sum_{j=1}^{d} \Vert\partial_{x_{j}} h_{\varepsilon}-\partial_{x_{j}}\Tilde{h}_{\varepsilon}\Vert_{L^{\infty}} \Vert\partial_{x_j}u_{\varepsilon}\Vert_{L^2} + \Vert h_{\varepsilon}-\Tilde{h}_{\varepsilon}\Vert_{L^{\infty}}\Vert\sum_{j=1}^{d} \partial_{x_j}^{2}u_{\varepsilon}\Vert_{L^2}.
\end{align*}
In the last inequality, we used the product rule for derivatives and the fact that $\Vert\partial_{x_{j}}\left(h_{\varepsilon}-\Tilde{h}_{\varepsilon}\right)\partial_{x_j}u_{\varepsilon}\Vert_{L^2}$ and $\Vert\left(h_{\varepsilon}-\Tilde{h}_{\varepsilon}\right)\partial_{x_j}^{2}u_{\varepsilon}\Vert_{L^2}$ can be estimated by\\
$\Vert\partial_{x_{j}} h_{\varepsilon}-\partial_{x_{j}}\Tilde{h}_{\varepsilon}\Vert_{L^{\infty}} \Vert\partial_{x_j}u_{\varepsilon}\Vert_{L^2}$ and $\Vert h_{\varepsilon}-\Tilde{h}_{\varepsilon}\Vert_{L^{\infty}} \Vert\partial_{x_j}^{2}u_{\varepsilon}\Vert_{L^2}$, respectively, for all $j=1,...d$. On the one hand, we have by assumption that
\begin{equation*}
    \Vert h_{\varepsilon}-\Tilde{h}_{\varepsilon}\Vert_{W^{1,\infty}}\leq C_{k}\varepsilon^{k} \text{\,\,for all\,\,} k>0,
\end{equation*}
and
\begin{equation*}
    \Vert u_{0,\varepsilon}-\Tilde{u}_{0,\varepsilon}\Vert_{L^{2}}\leq C_{m}\varepsilon^{m} \text{\,\,for all\,\,} m>0,
\end{equation*}
On the other hand, the net $(u_{\varepsilon})_{\varepsilon}$ is $C$-moderate as a very weak solution to the Cauchy problem (\ref{Equation singular}). Then, there exists $N\in \mathbb{N}$ such that
\begin{equation*}
    \Vert\partial_{x_i}u_{\varepsilon}\Vert_{L^2} \lesssim \varepsilon^{-N},
\end{equation*}
for all $i=1,...,d$, and
\begin{equation*}
    \Vert\sum_{j=1}^{d} \partial_{x_j}^{2}u_{\varepsilon}\Vert_{L^2} \lesssim \varepsilon^{-N}.
\end{equation*}
It follows that
\begin{equation*}
    \Vert U(\cdot,t)\Vert_{L^2} \lesssim \varepsilon^{n},
\end{equation*}
for all $n\in \mathbb{N}$ and the uniqueness is proved.
\end{proof}

\subsection{Consistency with the classical solution}
We conclude this paper by showing that in the case when the coefficient $h$ and the Cauchy data $u_0$ are regular enough in such way that a classical solution exists, then the very weak solution converges to the classical one in an appropriate norm.

\begin{thm}[Consistency] \label{thm consistency}
Let $h\in L^{\infty}(\mathbb{R}^d)$ satisfy $\inf_{x\in\mathbb{R}^d}h(x)>0$ and $\nabla h\in L^{\infty}(\mathbb{R}^d)$. Assume that $u_0 \in H^2(\mathbb{R}^d)$ and let us consider the Cauchy problem
\begin{equation}
    \left\lbrace
    \begin{array}{l}
    u_{t}(t,x) - \sum_{j=1}^{d}\partial_{x_{j}}\left(h(x)\partial_{x_{j}}u(t,x)\right)=0, \,\,\,(t,x)\in\left[0,T\right]\times \mathbb{R}^{d},\\
    u(0,x)=u_{0}(x), \,\,\, x\in\mathbb{R}^{d}. \label{Equation consistency}
    \end{array}
    \right.
\end{equation}
Let $(u_{\varepsilon})_{\varepsilon}$ be a very weak solution of (\ref{Equation consistency}). Then, for any regularising families $(h_{\varepsilon})_{\varepsilon}=(h\ast\psi_{\varepsilon})_{\varepsilon}$ and $(u_{0,\varepsilon})_{\varepsilon}=(u_{0}\ast\psi_{\varepsilon})_{\varepsilon}$ for any $\psi\in C_{0}^{\infty}$, $\psi\geq 0$, $\int\psi =1$, such that
\begin{equation*}
    \Vert h_{\varepsilon} - h\Vert_{W^{1,\infty}} \rightarrow 0,
\end{equation*}
the net $(u_{\varepsilon})_{\varepsilon}$ converges to the classical solution of the Cauchy problem (\ref{Equation consistency}) in $L^{2}$ as $\varepsilon \rightarrow 0$.
\end{thm}

\begin{proof}
Let $u$ be the classical solution. It solves
\begin{equation}
    \left\lbrace
    \begin{array}{l}
    u_{t}(t,x) - \sum_{j=1}^{d}\partial_{x_{j}}\left(h(x)\partial_{x_{j}}u(t,x)\right)=0, \,\,\,(t,x)\in\left[0,T\right]\times \mathbb{R}^{d},\\
    u(0,x)=u_{0}(x), \,\,\, x\in\mathbb{R}^{d},
    \end{array}
    \right.
\end{equation}
and let $(u_{\varepsilon})_{\varepsilon}$ be the very weak solution. It satisfies
\begin{equation}
    \left\lbrace
    \begin{array}{l}
    \partial_{t}u_{\varepsilon}(t,x) - \sum_{j=1}^{d}\partial_{x_{j}}\left(h_{\varepsilon}(x)\partial_{x_{j}}u_{\varepsilon}(t,x)\right)=0, \,\,\,(t,x)\in\left[0,T\right]\times \mathbb{R}^{d},\\
    u_{\varepsilon}(0,x)=u_{0,\varepsilon}(x), \,\,\, x\in\mathbb{R}^{d}.
    \end{array}
    \right.
\end{equation}
Let us denote by $V_{\varepsilon}(t,x):=u_{\varepsilon}(t,x)-u(t,x)$. Then $V_{\varepsilon}$ solves the problem
\begin{equation}
    \left\lbrace
    \begin{array}{l}
    \partial_{t}V_{\varepsilon}(t,x) - \sum_{j=1}^{d}\partial_{x_{j}}\left(h_{\varepsilon}(x)\partial_{x_{j}}V_{\varepsilon}(t,x)\right) = \eta_{\varepsilon}(t,x), \,\,\,(t,x)\in\left[0,T\right]\times \mathbb{R}^{d},\\
    V_{\varepsilon}(0,x)=(u_{0,\varepsilon}-u_{0})(x), \,\,\, x\in\mathbb{R}^{d},
    \end{array}
    \right.
\end{equation}
where
\begin{equation}
    \eta_{\varepsilon}(t,x):= \sum_{j=1}^{d}\partial_{x_{j}}\left[\left(h_{\varepsilon}(x)-h(x)\right)\partial_{x_{j}}u(t,x)\right].
\end{equation}
Arguing as in the proof of Theorem \ref{thm uniqueness}, we get the following estimate
\begin{align*}
    \Vert V_{\varepsilon}(\cdot,t)\Vert_{L^2} \lesssim \Vert u_{0,\varepsilon}-u_{0}\Vert_{L^2} & +  \sum_{j=1}^{d} \Vert\partial_{x_{j}} h_{\varepsilon}-\partial_{x_{j}}h\Vert_{L^{\infty}} \int_{0}^{T}\Vert\partial_{x_j}u(s,\cdot)\Vert_{L^2} ds\\
    & + \Vert h_{\varepsilon}-h\Vert_{L^{\infty}}\int_{0}^{T} \Vert\Delta_{x} u(s,\cdot)\Vert_{L^2} ds.
\end{align*}
We have that $\Vert h_{\varepsilon} - h\Vert_{W^{1,\infty}} \rightarrow 0$ and $\Vert u_{0,\varepsilon}-u_{0}\Vert_{L^2} \rightarrow 0$ as $\varepsilon\rightarrow 0$. On the other hand, $\partial_{x_j}u$, for all $j=1,...,d$ and $\Delta_{x} u$ are bounded in $L^2$ since $u$ is a classical solution. It follows that $(u_{\varepsilon})_{\varepsilon}$ converges in $L^2$ to the classical solution.
\end{proof}

\end{document}